\documentclass[11pt]{article}

\usepackage[T1]{fontenc}
\usepackage[utf8]{inputenc}
\usepackage[english]{babel}
\usepackage[nottoc,notlot,notlof]{tocbibind}
\usepackage{enumerate}
\usepackage{multirow,booktabs}
\usepackage[table]{xcolor}
\usepackage{fullpage}
\usepackage{lastpage}
\usepackage{indentfirst}
\usepackage{fancyhdr}
\usepackage{mathrsfs}
\usepackage{wrapfig}
\usepackage{setspace}
\usepackage{calc}
\usepackage{multicol}
\usepackage{cancel}
\usepackage[retainorgcmds]{IEEEtrantools}
\usepackage[margin=3cm]{geometry}

\usepackage{empheq}
\usepackage{framed}
\usepackage[most]{tcolorbox}
\usepackage{xcolor}
\colorlet{shadecolor}{orange!15}
\usepackage{footnote}

\usepackage{amsmath,amsfonts,amssymb,amscd}

\usepackage{bm}

\usepackage[all,2cell]{xy} \UseAllTwocells \SilentMatrices

\usepackage{hyperref}

\usepackage{subfiles}

\usepackage{amsthm}
\theoremstyle{definition}
\newtheorem{teo}{Theorem}[section] 
\newtheorem{cor}[teo]{Corollary} 

\newtheorem{prop}[teo]{Proposition} 

\newtheorem{defn}[teo]{Definition} 
\newtheorem{ex}[teo]{Example}

\usepackage{imakeidx}
\makeindex[intoc]

\usepackage{tikz}
\usetikzlibrary{positioning}

\tikzset{onslide/.code args={<#1>#2}{%
  \only<#1>{\pgfkeysalso{#2}} 
}}
\tikzset{temporal/.code args={<#1>#2#3#4}{%
  \temporal<#1>{\pgfkeysalso{#2}}{\pgfkeysalso{#3}}{\pgfkeysalso{#4}} 
the path
}}

  \tikzset{
    invisible/.style={opacity=0},
    visible on/.style={alt={#1{}{invisible}}},
    alt/.code args={<#1>#2#3}{%
      \alt<#1>{\pgfkeysalso{#2}}{\pgfkeysalso{#3}} 
    },
  }
  
\tikzstyle{highlight}=[red,ultra thick]

\usetikzlibrary{mindmap,arrows,shapes.geometric}

\tikzstyle{arrow} = [thick,->,>=stealth]

\usetikzlibrary{positioning}

\tikzset{onslide/.code args={<#1>#2}{%
  \only<#1>{\pgfkeysalso{#2}} 
}}
\tikzset{temporal/.code args={<#1>#2#3#4}{%
  \temporal<#1>{\pgfkeysalso{#2}}{\pgfkeysalso{#3}}{\pgfkeysalso{#4}} 
the path
}}

\tikzstyle{highlight}=[red,ultra thick]

\usetikzlibrary{shapes.geometric, arrows}
\tikzstyle{tp1} = [rectangle, rounded corners, minimum width=2cm, minimum height=1cm,text 
centered, text width=2cm, draw=black, fill=red!30]
\tikzstyle{tp2} = [rectangle, rounded corners, minimum width=2cm, 
minimum height=1cm, text centered, text width=2cm, draw=black, fill=blue!30]
\tikzstyle{tp3} = [rectangle, rounded corners, minimum width=2cm, minimum height=1cm, text 
centered, text width=2cm, draw=black, fill=green!30]
\tikzstyle{tp4} = [rectangle, rounded corners, minimum width=2cm, minimum height=1cm, text 
centered, text width=2cm, draw=black, fill=green!30]
\tikzstyle{tp5} = [rectangle, rounded corners, minimum width=4cm, minimum height=1cm, text 
centered, text width=5cm, draw=black, fill=blue!30]
\tikzstyle{arrow} = [thick,->,>=stealth]

\begin{document}

\title{Non Reduced Theory of Quadratic Forms Over Rings}
\author{Kaique Matias de Andrade Roberto \& Hugo Luiz Mariano}
\date{}

\maketitle

\tableofcontents

\thispagestyle{empty}


\begin{abstract}
 Based on the works of M. Marshall on multirings, we propose the fundamentals for a \textbf{non reduced}
abstract quadratic forms theory in general coefficients on rings, with the machinery of multirings and multifields.
\end{abstract}

\section{Introduction}

Quadratic forms over fields (of characteristc not 2) is a large and succesfull subject of study in mathematics. The book 
\cite{lam2005introduction} covers almost all introductory aspects of the theory as been as \cite{lam1983orderings} covers the 
relationship between orderings, valuations and quadratic forms, in the theory that will be known as Reduced Theory of Quadratic 
Forms. My master degree thesis \cite{roberto2019multiringchamber} 
establishes a path between the both classic and reduced theory with the abstract ones.

However, the generalization of the theory to general coefficients in rings is a hard step. The book \cite{knus1991quadratic} 
cover some basic aspects in the most general setting possible, and we have Marshall's theory of abstract 
real spectrum (\cite{marshall1996spaces}) and its algebraic counterpart, the realsemigroups of Dickmann and Petrovich 
(\cite{dickmann2004real}) given a nice approach for the \textbf{reduced} theory of quadratic forms on rings, but, most of the 
relevant aspects of quadratic forms, like Witt rings, Pfister forms and etc, are 
uncovered.

Based on the works of M. Marshall in the paper \cite{marshall2006real}, we propose the fundamentals for a \textbf{non reduced}
abstract quadratic forms theory in general coefficients on rings, with the machinery of multirings and multifields.

\section{Connections Between ARS, Realsemigroups and Multirings}

\begin{defn}[Abstract Real Spectra]\label{defn:ars}
 An \textit{abstract real spectrum or space of signs}\index{space of signs}, abreviatted to ARS, is 
a pair $(X,G)$ satisfying:
 \begin{description}
  \item [AX1 -] $X$ is a non-empty set, $G$ is a submonoid of $\{-1,0,1\}^X$, $G$ contais the constants 
functions $-1,0,1$, and $G$ separates points in $X$. 
 \end{description}

 If $a,b\in G$, the \textit{value set} $D(a,b)$ is defined to be the set of all $c\in G$ such that, for 
all $x\in X$, either $a(x)c(x)>0$ or $b(x)c(x)>0$ or $c(x)=0$. The \textit{value set} $D^t(a,b)$ is 
defined to be the set of all $c\in G$ such that, for all $x\in X$, either $a(x)c(x)>0$ or $b(x)c(x)>0$ 
or $c(x)=0$ and $b(x)=-a(x)$. Note that $c\in D^t(a,b)\Rightarrow c\in D(a,b)$. Conversely, $c\in 
D(a,b)\Rightarrow c\in D^t(ac^2,bc^2)$.

\begin{description}
 \item [AX2 -] If $P$ is a submonoid of $G$ satisfying $P\cup-P=G$, $-1\notin P$, $a,b\in P\Rightarrow 
D(a,b)\subseteq P$ and $ab\in P\cap-P\Rightarrow a\in P\cap-P$ or $b\in P\cap-P$, then there exists 
$x\in X$ (necessarily unique) such that $P=\{a\in G:a(x)\le0\}$.

\item [AX3a (Weak Associativity) -] For all $a,b,c\in G$, if $p\in D(a,r)$ for some $q\in D(b,c)$ 
then $p\in D(r,c)$ for some $r\in D(a,b)$.

\item [AX3b -] For all $a,b\in G$, $D^t(a,b)\ne\emptyset$.
\end{description}

We hasten to point out that AX3a and AX3b combined are equivalent to the simgle axiom AX3 below:

\begin{description}
\item [AX3 (Strong Associativity) -] For all $a,b,c\in G$, if $p\in D^t(a,r)$ for some $q\in D^t(b,c)$ 
then $p\in D^t(r,c)$ for some $r\in D^t(a,b)$.
\end{description}
\end{defn}

\begin{defn}\label{I.1.1rs}
A \textit{ternary semigroup} (abbreviated TS) is a struture $(S,\cdot,1,0,-1)$ with individual 
constants $1,,0,-1$ and a binary operation ``$\cdot$'' such that:
\begin{description}
 \item [TS1 -] $(S,\cdot,1)$ is a commutative semigroup with unity;
 \item [TS2 -] $x^3=x$ for all $x\in S$;
 \item [TS3 -] $-1\ne1$ and $(-1)(-1)=1$;
 \item [TS4 -] $x\cdot0=0$ for all $x\in S$;
 \item [TS5 -] For all $x\in S$, $x=-1\cdot x\Rightarrow x=0$.
\end{description}

We shall write $-x$ for $(-1)\cdot x$. The semigroup verifying conditions [TS1] and [TS2] (no extra 
constants) will be called \textit{3-semigroups}. 
\end{defn}

Here, we will enrich the language $\{\cdot,1,0,-1\}$ with a ternary relation $D$. In agreement with 
\ref{defn:ars}, we shall write $a\in D(b,c)$ instead of $D(a,b,c)$. We also set:
\begin{align*}
 \tag{trans}a\in D^t(b,c)\Leftrightarrow a\in D(b,c)\wedge-b\in D(-a,c)\wedge -c\in D(b,-a).
\end{align*}
The relations $D$ and $D^t$ are called \textit{representation} and \textit{transversal representation} 
respectivel.

\begin{defn}\label{I.2.2rs}
 A \textit{real semigroup} (abbreviated RS) is a ternary semigroup $(G,1,0,-1)$ together with a ternary 
relation $D$ satisfying:
\begin{description}
 \item [RS0 -] $c\in D(a,b)$ if and only if $c\in D(b,a)$.
 \item [RS1 -] $a\in D(a,b)$.
 \item [RS2 -] $a\in D(b,c)$ implies $ad\in D(bd,cd)$.
 \item [RS3 (Strong Associativity) -] If $a\in D^t(b,c)$ and $c\in D^t(d,e)$, then there exists $x\in 
D^t(b,d)$ such that $a\in D^t(x,e)$.
 \item [RS4 -] $e\in D(c^2a,d^2b)$ implies $e\in D(a,b)$.
 \item [RS5 -] If $ad=bd$, $ae=be$ and $c\in D(d,e)$, then $ac=bc$.
 \item [RS6 -] $c\in D(a,b)$ implies $c\in D^t(c^2a,c^2b)$.
 \item [RS7 (Reduction) -] $D^t(a,-b)\cap D^t(b,-a)\ne\empty$ implies $a=b$.
 \item [RS8 -] $a\in D(b,c)$ implies $a^2\in D(b^2,c^2)$.
\end{description}
A \textit{pre-real semigroup}\index{pre-real semigroup}(abbreviated PRS) is a ternary semigroup $(G,1,0,-1)$ together with a 
ternary 
relation $D$ satisfying [RS0]-[RS6], [RS8] and
\begin{description}
 \item [RS7' -] $x\in D^t(0,a)\Leftrightarrow x=a$.
\end{description}
\end{defn}

\begin{defn}\label{I.2.2rs2}
 A \textit{pre real semigroup} (abbreviated PRS) is a ternary semigroup $(G,1,0,-1)$ together with a ternary 
relation $D^t$ satisfying:
\begin{description}
 \item [DT0 -] $a\in D^t(b,c)$ if and only if $a\in D^t(c,b)$.
 \item [DT1 -] $a\in D^t(b,c)$ implies $-b\in D^t(-a,c)$.
 \item [DT2 -] $1\in D^t(1,a)$ for all $a\in G$.
 \item [DT3 -] $a\in D^t(b,c)$ implies $ad\in D^t(bd,cd)$.
 \item [DT4 (Strong Associativity) -] If $a\in D^t(b,c)$ and $c\in D^t(d,e)$, then there exists $x\in 
D^t(b,d)$ such that $a\in D^t(x,e)$.
 \item [DT5 -] If $ad=bd$, $ae=be$ and $c\in D^t(c^2d,c^2e)$, then $ac=bc$.
 \item [DT6 -] $e\in D^t(c^2e^2a,d^2e^2b)$ implies $e\in D^t(e^2a,e^2b)$.
 \item [DT7 -] $c\in D(a,b)$ implies $c\in D^t(c^2a,c^2b)$.
 \item [DT8 -] $a\in D^t(a^2b,a^2c)$ implies $a^2\in D^t((ab)^2,(ac)^2)$.
 \item [DT9 -] $x\in D^t(0,a)\Leftrightarrow x=a$.
\end{description}
A \textit{real semigroup} (abbreviated RS) is a pre-real semigroup satisfying
\begin{description}
 \item [DT10 (Reduction) -] $D^t(a,-b)\cap D^t(b,-a)\ne\emptyset$ implies $a=b$.
\end{description}
\end{defn}

\begin{teo}[The Duality Theorem]\label{4.1dickmann}
 There is a functorial duality between the category $\mathcal{RS}$ of real semigroups with RS-morphisms and $\mathcal{ARS}$ of 
abstract real spectra with ARS-morphisms. Moreover, the duality establishes an equivalence between the categories $\mathcal{RS}$ 
and $\mathcal{ARS}^{op}$, the opposite category of $\mathcal{ARS}$.
\end{teo}

Let $A$ be a multiring with $-1 \notin \sum A^2$ and consider $T$ a preorder. We denote by $Q_T (A)$ the image of $A$ in $Q_2 ^ 
{X_T}$, where $X_T = \{P \in \mbox{Sper}(A): T \subseteq P\}$.

Denote the image of $A$ in $Q_2^{X_T}$ by $Q_T(A)$. Addition on $Q_T(A)$ is defined by
$\overline a+\overline b:=\{\overline c:c\in a+b\}$, $\overline a\overline b:=\overline{ab}$, 
$-\overline{a}:=\overline{-a}$. The zero element of $Q_T(A)$ is $\overline{0}$.

\begin{prop}[Local-Global principle]\label{prop:7.3marshall}
 Let $A$ be a multiring with $-1\notin\sum A^2$ and $T$ a proper preordering of $A$. Then:
 \begin{enumerate}
  \item $Q_T(A)$ is a multiring.
  \item $Q_T(A)$ is strong embedded in $Q_2^{X_T}$.
 \end{enumerate}
\end{prop}

We denote $Q_{\sum A^2}(A)$ by $Q_{red}(A)$ which we refer to as the \textit{real reduced multiring} associated to $A$.

\begin{prop}\label{prop:7.5marshall}
 For a multiring $A$ with $-1\notin\sum A^2$, the map $a\mapsto\overline a$ from $A$ onto $Q_{red}(A)$ 
is an isomorphism if and only if $A$ satisfies the following properties:
\begin{enumerate}[a -]
 \item $a^3=a$.
 \item $a+ab^2=\{a\}$.
 \item $a^2+b^2$ contains a unique element.
\end{enumerate}
\end{prop}

\begin{defn}\label{defn:mrrealreduced}
 A multiring satisfying $-1\notin\sum A^2$ and the equivalent conditions of proposition 
\ref{prop:7.5marshall} will be called \textit{real reduced multiring}. A morphism of real reduced 
multirings is just a morphism of multirings. The category of real 
reduced multirings will be denoted by $\mathcal{MR}_{red}$.
\end{defn}

\begin{cor}\label{cor:7.6marshall}
 A multiring $A$ is real reduced if and only if the following properties holds for all 
$a,b,c,d\in A$:
\begin{enumerate}[i -]
 \item $1\ne0$;
 \item $a^3=a$;
 \item $c\in a+ab^2\Rightarrow c=a$;
 \item $c\in a^2+b^2$ and $d\in a^2+b^2$ implies $c=d$.
\end{enumerate}
\end{cor}

This implies that the morphism $a\mapsto\overline a$ from $A$ to $Q_{\mbox{red}}(A)$ is an 
isomorphism. In particular, follow by the local-global principle \ref{prop:7.3marshall} that 
for any real reduced multiring $A$, $c\in a+b\subseteq F$ if and only if, 
for every $\sigma:A\rightarrow Q_2$, $\sigma(c)\in\sigma(a)+\sigma(b)$.

\begin{prop}
   Let $A$ be a real reduced multiring. Then we have the following:
\begin{enumerate}[i -]
  \item $x\in a x^2 + b x^2$ if and only if $x \in a A^2 + b A^2$;
  \item $x\in a+b$ if and only if $x\in a x^2 + b x^2$, $-a \in b a^2 - x a^2$ and $-b\in a b^2 - x b^2$;
  \item If $ax = bx$, $ay = by$ and $z \in x z^2 + y z^2$, then $az = bz$;
  \item If $x \in a x^2 + b x^2$, then $x^2 \in a^2 x^2 + b^2 x^2$.
\end{enumerate}
\end{prop}

 \begin{teo}\label{teo:arstomrred}
 Let $(X,G)$ an abstract real spectra and define $a+b=\{d\in G:d\in D^t(a,b)\}$. Then 
$(G,+,\cdot,-,0,1)$ is a real reduced multiring\index{multiring!real reduced}.
\end{teo}

\begin{teo}\label{teo:mrredtoars}
 Let $A$ be an real reduced multiring and consider the strong embedding $i:A\rightarrow 
Q_2^{\mbox{Sper}(A)}$ given by $i(a)=\hat a:\mbox{Sper}(A)\rightarrow Q_2$ when $\hat a(\sigma)=\sigma 
(a)$. Define $\hat A=i(A)$. Then $(\mbox{Sper}(A),\hat A)$ is an abstract real spectra. 
\end{teo}

\begin{teo}\label{teo:rstomrred}
 Let $(G,\cdot,1,0,-1,D)$ be a realsemigroup and define $+:G\times 
G\rightarrow\mathcal{P}(G)\setminus\{\emptyset\}$, $a+b=\{d\in G:d\in D^t(a,b)\}$ and $-:G\rightarrow 
G$ by $-(g)=-1\cdot g$. Then $(G,+,\cdot,-,0,1)$ is a real reduced multiring.
\end{teo}

\begin{teo}\label{teo:mrredtors}
 Let $A$ be a real reduced multiring. Then $(A,\cdot,1,0,-1,D)$ is a realsemigroup, where $d\in D(a,b)\Leftrightarrow d\in 
d^2a+d^2b$.
\end{teo}

\section{Quadratically Tuned Multirings}

\begin{defn}
 A multiring is said to be \textbf{quadratically tuned} if 
 for all $a,b,c,d,e\in A$:
 \begin{description}
  \item [QT 0 -] $a^3=a$.
  
  \item [QT 1 -] $1\in 1+a$.
  
  \item [QT 2 -] If $ad=bd$, $ae=be$ and $c\in c^2d+c^2e$ then $ac=bc$.
  
  \item [QT 3 -] If $e\in (ce)^2a+(de)^2b$ then $e\in e^2a+e^2b$.
  
  \item [QT 4 -] If $a\in a^2b+a^2c$ then $a^2\in(ab)^2+(ac)^2$.
  
  \item [QT 5 -] $e\in a+b$ if and only if $e\in a e^2 + b e^2$, $-a \in b a^2 - e a^2$ and $-b\in a b^2 - e b^2$.
 \end{description}
 A morphism between two quadratically tuned multirings is just a multiring morphism. The category of quadratically tuned 
multirings will be denoted by $\mathcal{MR}_{qt}$.
\end{defn}

\begin{prop}\label{prop1}
 Every real reduced multiring is a quadratically tuned multiring. In particular, every real reduced multifield is a quadratically 
tuned multiring
\end{prop}

\begin{prop}\label{prop2}
 Every special multifield is a quadratically tuned multiring.
\end{prop}

Set $d\in D(a,b)\Leftrightarrow d\in ad^2+bd^2$.

\begin{prop}
 Let $A$ be a qt-multiring and $a,b,c,d,e\in A$.
 \begin{enumerate}[i -]
  \item $d\in D(a,b)$ iff $d\in aA^2+bA^2$.
  
%
 \end{enumerate}
\end{prop}
\begin{proof}
  \begin{enumerate}[i -]
  \item ($\Rightarrow$) is just the definition of $D$. For ($\Leftarrow$), let $d\in ax^2+by^2$. Then
  \begin{align*}
   d\in ax^2+by^2&\Rightarrow d^3=d\in ax^2d^2+by^2d^2 \\
   &\stackrel{QT3}{\Rightarrow}d\in ad^2+bd^2\Rightarrow d\in D(a,b).
  \end{align*}
  
%
 \end{enumerate}
\end{proof}

\begin{teo}
 Let $(G,\cdot,1,0,-1,D)$ be a pre-real semigroup and define $+:G\times 
G\rightarrow\mathcal{P}(G)\setminus\{\emptyset\}$, $a+b=\{d\in G:d\in D^t(a,b)\}$ and $-:G\rightarrow 
G$ by $-(g)=-1\cdot g$. Then $(G,+,\cdot,-,0,1)$ is a quadratically tuned multiring.
\end{teo}

\begin{cor}
 The correspondence above can be extended to a functor $M:\mathcal{PRS}\rightarrow\mathcal{MR}_{qt}$.
\end{cor}

\begin{teo}
 Let $A$ be a quadratically tuned multiring. Then $(A,\cdot,1,0,-1,D)$ is a pre-real semigroup, where $d\in D(a,b)\Leftrightarrow 
d\in 
d^2a+d^2b$.
\end{teo}

\begin{cor}
 The correspondence above can be extended to a functor $R:\mathcal{MR}_{qt}\rightarrow\mathcal{PRS}$.
\end{cor}

\begin{cor}
 The following diagram commute:
 $$
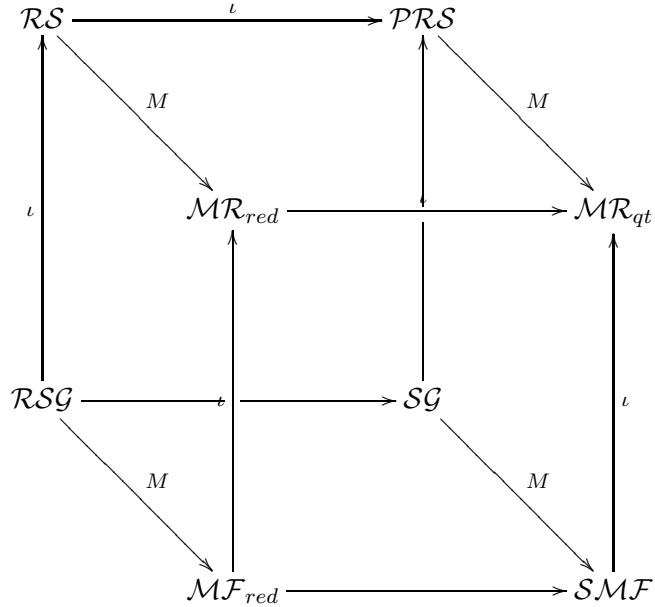
\xymatrix@!=4.pc{\mathcal{RS}\ar[dr]^M\ar[rr]^\iota && \mathcal{PRS}\ar[dr]^M \\
 & \mathcal{MR}_{red}\ar[rr]^\iota && \mathcal{MR}_{qt} \\
 \mathcal{RSG}\ar[rd]^M\ar'[r][rr]\ar[uu]^\iota && \mathcal{SG}\ar[rd]^M\ar'[u][uu] \\
 & \mathcal{MF}_{red}\ar[uu]^\iota\ar[rr] && \mathcal{SMF}\ar[uu]_\iota}$$
\end{cor}

\section{Examples}

We already see that special multifields, and real reduced multifields and multirings are quadratically tuned multirings 
(\ref{prop1}, \ref{prop2}). By the functors established in \cite{ribeiro2016functorial}, we have in particular, that special 
groups and real semigroups are quadratically tuned multirings. Now, in the view of \cite{roberto2016promiscuously}, we have a 
more ``algebraically tasted'' example below. 

\begin{ex}
 Let $A$ be a ring with $2\in\dot A$ and $T$ a preordering. Set $S=T\setminus\{0\}$ and
 $$M(A):=A/_mS \mbox{(The Marshall's Quotient)}.$$
 Since $\overline a^2=\overline 1\in M(A)$, is straghtforward verify that $M(A)$ is a quadratically tuned multiring. Now, as in 
\cite{dickmann2004real}, let the set 
$G_T(A)$ consist of all functions $\overline a:\mbox{Sper}_T(A)\rightarrow\bm3$, for $a\in A$, where
$$\overline a(\alpha)=\begin{cases}
                          1\mbox{ if }a\in\alpha\setminus(-\alpha) \\
                          0\mbox{ if }a\in\alpha\cap-\alpha \\
                          -1\mbox{ if }a\in(-\alpha)\cap\alpha.
                         \end{cases}$$
with the operation induced by product in $A$. We have that $G_T(A)$ is a real semigroup. The map $a\mapsto\overline a$ from 
$M(A)$ to $G_T(A)$ is a multiring morphism. Since that for all $t\in T\setminus\{0\}$, $\overline t=\overline1\in G_T(A)$, by the 
universal property of Marshall's quotient we have an induced (surjective) morphism $\varphi:M(A)\rightarrow G_T(A)$.
 \end{ex}

\section{Quadratic Forms over QT-Multirings}

Let $A$ be a QT-multiring which will remain fixed throughout this section.

A \textbf{form of dimension $n\ge1$ over $A$} is just an $n$-tuple $\varphi:=\langle a_1,...,a_n\rangle$ where 
$a_1,...,a_n\in A$. The dimension of $\varphi$ is denoted by $\mbox{dim}(\varphi)$. The discriminant of $f$ is 
defined to be $\mbox{disc}(\varphi):=a_1a_2...a_n\in A$. If $a\in A$, we can \textbf{scale} $\varphi$ by $a$ to 
obtain the form $a\varphi:=\langle aa_1,...,aa_n\rangle$. The \textit{sum} of $\varphi$ and a form $\psi=\langle 
b_1,...,b_m\rangle$ is defined by $\varphi\oplus\psi=\langle 
a_1,...,a_n,b_1,...,b_m\rangle$ and the \textbf{tensor product} of $\varphi$ and $\psi$ is defined by
$\varphi\otimes\psi=\langle a_1b_1,...,a_ib_j,...,a_nb_m\rangle$.

\textbf{Strong Isometry} of one and two-dimensional forms is defined by $\langle a\rangle\equiv_s\langle 
b\rangle\Leftrightarrow a=b$ and $\langle a,b\rangle\equiv\langle c,d\rangle\Leftrightarrow$ $ab=cd$ 
and $a+b=c+d$. For forms of dimension $n\ge3$ isometry is defined inductively by:
$\langle a_1,...,a_n\rangle\equiv_s\langle b_1,...,b_n\rangle$ if and only if there are 
 $x,y,z_3,...,z_n\in A$ such that $\langle a_1,x\rangle\equiv_s\langle b_1,y\rangle$,
 $\langle a_2,...,a_n\rangle\equiv_s\langle x,z_3,...,z_n\rangle$ and
 $\langle b_2,...,b_n\rangle\equiv_s\langle y,z_3,...,z_n\rangle$.

 The strong isometry is just the extension of the isometry concepts provenient of the already known abstract theories of 
quadratic forms (such as special groups and special multifields). However, in the multiring context, we have another 
``natural'' isometry at hand, the \textbf{weak isometry} defined below:
 
\textbf{Weak Isometry} of one and two-dimensional forms is defined by $\langle a\rangle\equiv_w\langle 
b\rangle\Leftrightarrow a=b$ and $\langle a,b\rangle\equiv\langle c,d\rangle\Leftrightarrow$ $ab=cd$ 
and $a+b\cap c+d\ne\emptyset$. For forms of dimension $n\ge3$ isometry is defined inductively by:
$\langle a_1,...,a_n\rangle\equiv_w\langle b_1,...,b_n\rangle$ if and only if there are 
 $x,y,z_3,...,z_n\in A$ such that $\langle a_1,x\rangle\equiv_w\langle b_1,y\rangle$,
 $\langle a_2,...,a_n\rangle\equiv_w\langle x,z_3,...,z_n\rangle$ and
 $\langle b_2,...,b_n\rangle\equiv_w\langle y,z_3,...,z_n\rangle$. 
 
 Observe that $\varphi\equiv_s\psi\Rightarrow\varphi\equiv_w\psi$.

 \begin{prop}
  Let $A$ be a QT-multiring, $a,b,c,d,e\in A$ and $\varphi:=\langle a_1,...,a_n\rangle$ and $\psi=\langle 
b_1,...,b_m\rangle$ be forms over $A$. For $\sigma\in S_n$, write $\varphi^{\sigma}$ for the $n$-form
$\varphi^{\sigma}=\langle a_{\sigma(1)},...,a_{\sigma(n)}\rangle$. Then
\begin{enumerate}[i -]
 \item $\langle a,b,c\rangle\equiv_w\langle a,b,c\rangle^{\sigma}$.
 
 \item $\varphi\equiv_w\varphi^{\sigma}$.
 
%
%
\end{enumerate}
 \end{prop}
%
%
%
%

\newpage

\bibliographystyle{plain}
\bibliography{/home/kaique/Dropbox/Textos/one_for_all}


\end{document}